\documentclass[12pt]{article}
\usepackage{lineno,hyperref}
\usepackage{bbm}
\usepackage{float}
\usepackage{mathrsfs}
\usepackage{graphicx}
\usepackage[utf8]{inputenc}
\usepackage{amsmath}
\usepackage{amstext}
\usepackage{amsfonts}
\usepackage{amsthm}
\usepackage{amssymb}
\usepackage[bf,SL,BF]{subfigure}
\usepackage{subfigure}
\usepackage{caption}
\newcommand{\triple}{{\vert\kern-0.25ex\vert\kern-0.25ex\vert}}
\theoremstyle{plain}

\makeatletter
\newcommand*\bigcdot{\mathpalette\bigcdot@{.45}}
\newcommand*\bigcdot@[2]{\mathbin{\vcenter{\hbox{\scalebox{#2}{$\m@th#1\bullet$}}}}}
\makeatother
\newtheorem{definition}{Definition}[section]
\newtheorem{theorem}[definition]{Theorem}

\newtheorem{corollary}[definition]{Corollary}

\theoremstyle{definition}
\newtheorem{remark}[definition]{Remark}

\renewcommand\labelenumi{(\roman{enumi})}
\renewcommand\theenumi\labelenumi

\begin{document}
\title{\bf On the Analysis of a Generalised Rough Ait-Sahalia Interest Rate Model }
\author
{{\bf Emmanuel Coffie,\
\bf Xuerong Mao,\
\bf Frank Proske}}
\date{}
\maketitle
\begin{abstract}
Fractional Brownian motion with the Hurst parameter $H<\frac{1}{2}$ is used widely, for instance, to describe a 'rough' stochastic volatility process in finance. In this paper, we examine an Ait-Sahalia-type interest rate model driven by a fractional Brownian motion with $H<\frac{1}{2}$ and establish theoretical properties such as an existence-and-uniqueness theorem, regularity in the sense of Malliavin differentiability and higher moments of the strong solutions.
\end{abstract}
\section{Introduction}\label{sec1}
Over the years, SDEs driven by noise with $\alpha^-$-H\"older continuous random paths for $\alpha\in[\frac{1}{2},1)$ have been applied to model the dynamical behaviour of asset prices in finance. See e.g. \cite{Karatzas,Lamberton} and the references therein. However, in recent years, empirical evidence (see e.g. \cite{Gatheral}) has shown that volatility paths of asset prices are more irregular in the sense of $\alpha^-$-H\"older  continuity for $\alpha\in(0,\frac{1}{2})$ in many instances. This inadequacy actually showed the need for models based on SDEs driven by a noise of low $\alpha^-$-H\"older regularity with $\alpha\in(0,\frac{1}{2})$ which has been used by researchers and practitioners to describe the volatility dynamics of asset prices. These models are driven by rough signals that can capture well the 'roughness' in the volatility process of asset prices. Such rough signals arise e.g. from paths of the fractional Brownian motion (fBm). The fractional Brownian motion is a generalisation of the ordinary Brownian motion. It is a centred self-similar Gaussian process with stationary increments which depends on the Hurst parameter H. The Hurst parameter lies in $(0,1)$ and controls the regularity of the sample paths in the sense of a.e. (local) $H^-$-H\"older  continuity. The smaller the Hurst parameter, the rougher the sample paths and vice versa. For instance, the authors in \cite{amine} employ the fractional Brownian motion with $H<\frac{1}{2}$ to model the 'rough' volatility process of asset prices and derive a representation of the sensitivity parameter delta for option prices.  Similarly, the authors in \cite{Coffie} also consider an asset price model in connection with the sensitivity analysis of option prices whose correlated 'rough' volatility dynamics is described by means of a SDE driven by a fractional Brownian motion with $H<\frac{1}{2}$. The reader may consult \cite{Nualart, Biagini} for the coverage of properties and financial applications of the fractional Brownian motion with $H<\frac{1}{2}$ . See also the Appendix.
\par
In the context of interest rate modelling, Ait-Sahalia proposed a new class of highly nonlinear stochastic models in \cite{ait} for the evolution of interest rates through time after rejecting existing univariate linear-drift stochastic models based on empirical studies. In this model, (short term) interest rates $x_t$ have the SDE dynamics
\begin{equation}\label{intro:eq:1}
  dx(t)=(\alpha_{-1}x(t)^{-1}-\alpha_{0}+\alpha_{1}x(t)-\alpha_{2}x(t)^{2})dt+\sigma x(t)^{\theta}dB_t
\end{equation}
on $t\ge 0$ with initial value $x(0)=x_0$, where $\alpha_{-1},\alpha_{0},\alpha_{1}, \alpha_{2}>0$, $\sigma>0$, $\theta >1$ and $B_t$ is a scalar Brownian motion. This type of interest rate model has been studied by many authors (see e.g. \cite{Szpruch, Emma}).
\par
In order to capture "rough" (short term) interest rates, e.g. in turbulent bond markets, one may replace  the driving noise $B_{\bigcdot}$ in \eqref{intro:eq:1} by a fractional Brownian motion $B^H_{\bigcdot}$ and consider an interest rate  model based on the SDE
\begin{equation}\label{intro:eq:2}
  dx(t)=(\alpha_{-1}x(t)^{-1}-\alpha_{0}+\alpha_{1}x(t)-\alpha_{2}t^{2H-1}x(t)^{\rho})dt+\sigma x(t)^{\theta}dB_t^H
\end{equation}
for $t\ge 0$ with initial value $x(0)=x_0$, $t\in(0,1]$, $H\in(0,\frac{1}{2})$ and $\rho>1$. The stochastic integral for the fractional Brownian motion in \eqref{intro:eq:2} is defined via an integral concept in \cite{Nualart} and related to a Wick-It\^{o}-Skorohod type of integral. See also Section \ref{sec6}.
\par
On the other hand, an alternative model to \eqref{intro:eq:2} for the description of rough interest rate dynamics could be the following SDE, which "preserves" the classical drift structure of the Ait-Sahalia model in \eqref{intro:eq:1}:
\begin{equation}\label{intro:eq:3}
  dx(t)=(\alpha_{-1}x(t)^{-1}-\alpha_{0}+\alpha_{1}x(t)-\alpha_{2}x(t)^{\rho})dt+\sigma x(t)^{\theta}d^{\circ}B_t^H
\end{equation}
for $t\ge 0$ and $H\in(0,\frac{1}{2})$, where $\sigma x(t)^{\theta}d^{\circ}B_t^H$ stands for a stochastic integral in the sense of Russo and Vallois. See Section \ref{sec6}.
\par 
 Although we also obtain an existence and uniqueness result for solutions to \eqref{intro:eq:3} (see Theorem \ref{final}), we mainly focus in this paper on the study of SDE \eqref{intro:eq:2}. 
\par 
The remainder of the paper is organised as follows: In Section \ref{sec3}, we introduce the fractional Ait-Sahalia interest rate model. We establish an existence and uniqueness result for solutions to SDE \eqref{intro:eq:2} in Section \ref{sec6} by studying the properties of solutions to an associated SDE driven by an additive fractional noise (see Sections \ref{sec4} and \ref{sec5}). In addition, we also discuss the alternative model \eqref{intro:eq:3} in Section \ref{sec6}.
\section{The fractional Ait-Sahalia model}\label{sec3}
Throughout this paper unless specified otherwise, we employ the following notation. Let $(\Omega, \mathcal{F},\mathbb{P})$ be a complete probability space with filtration $\{ \mathcal{F}_t\}_{t\geq 0}$ satisfying the usual conditions (i.e., it is increasing and right continuous while $\mathcal{F}_0$ contains all $\mathbb{P}$ null sets). Denote $\mathbb{E}$ as the expectation corresponding to $\mathbb{P}$. Suppose that $B^H_t$, $0\le t\le 1$, is a scalar fractional Brownian motion (fBm) with Hurst parameter $H\in (0,\frac{1}{2})$ and $B_t$, $0\le t\le 1$, is a scalar Brownian motion defined on this probability space. 
\par
In what follows, we are interested to study the SDE 
\begin{equation}\label{chap1:eq:1}
x_t=x_0+\int_0^t\big(\alpha_{-1}x_s^{-1}-\alpha_0+\alpha_1x_s-\alpha_2s^{2H-1}x_s^{\rho}\big)ds+\int_0^t\sigma x^{\theta}_sdB_s^H,
\end{equation}
$x_0\in(0,\infty)$, $0\le t\le 1$, where $H\in(\frac{1}{3},\frac{1}{2})$, $\tilde{\theta}>0$, $\rho>1+\frac{1}{H\tilde{\theta}}$, $\theta:=\frac{\tilde{\theta}+1}{\tilde{\theta}}$, $\sigma>0$ and $\alpha_i>0$, $i=-1,\cdots,2$. Here the stochastic integral term with respect to $B_{\bigcdot}^H$ in \eqref{chap1:eq:1} is defined by means of an integral concept introduced by F. Russo, P. Vallois in \cite{Errami}. See Section \ref{sec6}. 
\par
As already mentioned in the introduction, solutions to the SDE \eqref{chap1:eq:1} can be used as a model (fractional Ait-Sahalia model) for the description of the dynamics of (short term) interest rates in finance. In fact, in this paper, we aim at establishing the existence and uniqueness of strong solutions $x_t>0$ to SDE \eqref{chap1:eq:1}. In doing so, we show that such solutions can be obtained as transformations to the SDE
\begin{equation}\label{sec2:eq1}
y_t=x+\int_0^t \tilde{f}(s,y_s)ds-\tilde{\sigma} B_t^H, \quad 0\le t\le 1, \quad H\in(0,\frac{1}{2}),
\end{equation}
where 
\begin{align}\label{sec2:eq2}
\tilde{f}(s,y)&=\alpha_{-1}(-\tilde{\theta} y^{2\tilde{\theta}+1})+\alpha_0y^{\tilde{\theta}+1}-\alpha_1\frac{y}{\tilde{\theta}}+\alpha_2 s^{2H-1}\frac{1}{\tilde{\theta}^{\rho}} y^{-\tilde{\theta}\rho+\tilde{\theta}+1}-\tilde{\sigma}Hs^{2H-1}y^{-1}(\tilde{\theta}+1),
\end{align}
where $\tilde{\sigma}>0$, $0< s\le 1$, $0<y<\infty$. See Section \ref{sec6}. 
\par
In the sequel, we want to prove the following new properties for solutions to SDE \eqref{sec2:eq1}:
\begin{itemize}
\item Existence and uniqueness of positive strong solutions (Corollary  \ref{thrm1})
\item Regularity of solutions in the sense of Malliavin differentiability (Theorem \ref{Thrm 4.2})
\item Existence of higher moments (Theorem \ref{highermoment}).
\end{itemize}

\section{Existence and uniqueness of solutions to singular SDEs with additive fractional noise for $H< \frac{1}{2}$}\label{sec4}
In this section, we wish to analyse the following generalisation of the SDE \eqref{sec2:eq1} given by
\begin{equation}\label{sec2:eq3}
x_t=x_0+\int^t_0b(s,x_s)ds+\sigma B^H_t,\quad 0\le t\le 1,\quad H\in(0,\frac{1}{2}),\quad \sigma>0.
\end{equation}
We require the following conditions
\begin{description}
\item (A1) $b\in C\big((0,1)\times (0,\infty)\big)$ and has a continuous spatial derivative $b^{\prime}:=\frac{\partial }{\partial x}b$ such that 
\begin{equation*}
b'(t,x)\le K_t, \quad 0< t< 1, \quad x\in(0,\infty),
\end{equation*}
where $K_t:=t^{2H-1}K$ for some $K\ge 0$.
\item (A2) There exist $x_1>0$, $\alpha>\frac{1}{H}-1$ and $h_1>0$ such that $b(t,x)\ge h_1t^{2H-1}x^{-\alpha}$, $t\in(0,1]$, $x\le x_1$.
\item (A3) There are $x_2>0$ and $h_2>0$ such that $b(t,x)\le h_2t^{2H-1}(x+1)$, $t\in(0,1]$, $x\ge x_2$.
\end{description}
\begin{theorem}\label{thrm1}
Suppose that \textup{(A1-A3)} hold. Then for all $x_0>0$ the \textup{SDE} \eqref{sec2:eq3} has a unique strong positive solution $x_t$, $0\le t\le1$.
\end{theorem}
\begin{proof}
Without loss of generality, let $\sigma=1$. We are required to establish the following analytical properties.
\begin{enumerate}
\item Uniqueness: Suppose $x_{\bigcdot}$ and $y_{\bigcdot}$ are two solutions to \eqref{sec2:eq3}. Then 
\begin{align*}
x_t-y_t=\int_0^t\big(b(s,x_s)-b(s,y_s)\big)ds.
\end{align*}
So, using the product rule, the mean value theorem and (A1), we get 

\begin{align*}
(x_t-y_t)^2&=2\int_0^t\big(b(s,x_s)-b(s,y_s)\big)(x_s-y_s)ds\\
&\le 2\int_0^tK_s(x_s-y_s)^2ds.
\end{align*}
Hence, Gronwall's Lemma implies that 
\begin{align*}
x_t-y_t=0,\quad 0\le t\le 1.
\end{align*}

\item Existence: Let $x_0>0$. Because of the regularity assumptions imposed on $b$, we know that the equation \eqref{sec2:eq3} has (path-by-path) local solutions. Define the stopping times
\begin{equation*}
\tau_0:=\inf\{t\in[0,1]:x_t=0\}
\text{ and }
\tau_n:=\inf\{t\in[0,1]:x_t\ge n\},
\end{equation*}
where $\inf\emptyset:=1^+$. Just as in \cite{Zhang}, we want to prove that $\tau_0=1^+$ and $\lim_{n\rightarrow \infty}\tau_n=1^+$. Here $1^+$ stands for an artificially added element larger than 1. Suppose that $\tau_0\le 1$. Then there is a $\hat{\tau}_0\in (0,\tau_0]$ such that $x_t\le x_1$ for all $(\hat{\tau}_0,\tau_0]$. By (A2), we know that $b(t,x)>0$ for $x\in(0,x_1)$ and $t>0$. Hence,
\begin{equation}
0=x_{\tau_0}=x_t+\int_t^{\tau_0}b(s,x_s)ds+B^H_{\tau_0}-B^H_t,\quad t\in (\hat{\tau}_0,\tau_0].
\end{equation}
This implies
\begin{equation}
x_t\le \vert B^H_{\tau_0}-B^H_t\vert\le \vert\vert B_{\bigcdot}^H\vert\vert_{\beta}(\tau_0-t)^{\beta},\ \text{$t\in (\hat{\tau}_0,\tau_0]$ for $\beta\in (0,H)$}.
\end{equation}
Here $\vert\vert \cdot\vert\vert_{\beta}$ denotes the H\"older-seminorm given by 
\begin{equation*}
\vert\vert f\vert\vert_{\beta}=\sup_{0\le s<t\le 1}\frac{\vert f(s)-f(t)\vert}{(t-s)^{\beta}}
\end{equation*}
for $\beta$-H\"older continuous functions $f$. So, we also obtain that
\begin{align*}
\vert\vert B_{\bigcdot}^H\vert\vert_{\beta}(\tau_0-t)^{\beta}&\ge \vert B^H_{\tau_0}-B^H_t\vert \ge \int^{\tau_0}_tb(s,x_s)ds\\
&\ge h_1\int^{\tau_0}_t s^{2H-1}x_s^{-\alpha}ds\ge \frac{h_1}{\vert\vert B_{\bigcdot}^H\vert\vert_{\beta}^{\alpha}}\int^{\tau_0}_t s^{2H-1}\frac{1}{(\tau_0-s)^{\alpha\beta}}ds\\
&\ge \frac{h_1}{\vert\vert B_{\bigcdot}^H\vert\vert_{\beta}^{\alpha}}\tau_0^{2H-1}\int^{\tau_0}_t \frac{1}{(\tau_0-s)^{\alpha\beta}}ds.
\end{align*}
\end{enumerate}
If $\alpha \beta\ge 1$, we get a contradiction. For $\alpha \beta< 1$, we find that
\begin{align*}
\vert\vert B_{\bigcdot}^H\vert\vert_{\beta}(\tau_0-t)^{\beta}&\ge \frac{h_1}{\vert\vert B_{\bigcdot}^H\vert\vert_{\beta}^{\alpha}}\tau_0^{2H-1} \frac{(\tau_0-t)^{1-\alpha\beta}}{1-\alpha\beta},\quad t\in (\hat{\tau}_0,\tau_0].
\end{align*}
Hence,
\begin{align*}
0=\lim_{t\rightarrow \tau_0}\vert\vert B_{\bigcdot}^H\vert\vert_{\beta}(\tau_0-t)^{\beta+\alpha\beta-1}\ge \frac{h_1\tau_0^{2H-1}}{\vert\vert B_{\bigcdot}^H\vert\vert_{\beta}^{\alpha}(1-\alpha\beta)}>0.
\end{align*}
So $\tau_0=1^+$. Assume now that 
\begin{align*}
\tau_{\infty}:=\lim_{n\rightarrow \infty}\tau_n\le 1.
\end{align*} 

Then we can show as in \cite{Zhang} by using (A3) that 
\begin{align*}
x_t\le x_2+x_0+\vert\vert B_{\bigcdot}^H\vert\vert_{\beta}\tau^{\beta}_{\infty}+h_2\tau^{2H}_{\infty}(2H)^{-1}+h_2\int_{\hat{\tau}_1}^t s^{2H-1}x_sds.
\end{align*}
So, by letting 
\begin{align*}
\alpha= x_2+x_0+\vert\vert B_{\bigcdot}^H\vert\vert_{\beta}\tau^{\beta}_{\infty}+h_2\tau^{2H}_{\infty}(2H)^{-1},
\end{align*}
it follows from Gronwall's Lemma that 
\begin{align*}
x_t&\le \alpha+\int_{\hat{\tau}_1}^t \alpha h_2s^{2H-1}\exp\Big(\int_s^th_2u^{2H-1}du\Big)ds\\
 &\le  \gamma+\int_0^1 \gamma h_2s^{2H-1}\exp\Big(\int_s^1 h_2u^{2H-1}du\Big)ds,
\end{align*}
where $\gamma:=x_2+x_0+\vert\vert B_{\bigcdot}^H\vert\vert_{\beta}+\frac{h_2}{2H}$. The latter estimate leads to a contradiction.
\end{proof}
As a consequence of Theorem \ref{thrm1}, we obtain the following result:
\begin{corollary}\label{corona}
Suppose that $x\in (0,\infty)$ and $\rho>\frac{1}{H\tilde{\theta}}+1$. Then there exists a unique strong solution $y_t>0$ to \textup{SDE} \eqref{sec2:eq1}.
\end{corollary}
\begin{proof}
Let $\epsilon=\frac{H}{2}$. Then 
\begin{equation*}
\tilde{f}(s,y)=\tilde{g}_1(s,y)+\tilde{g}_2(s,y),
\end{equation*}
where 
\begin{equation*}
\tilde{g}_1(s,y):=\alpha_{-1}(-\tilde{\theta} y^{2\tilde{\theta}+1})+\alpha_0y^{\tilde{\theta}+1}-\alpha_1\frac{y}{\tilde{\theta}}+\epsilon\tilde{\sigma}s^{2H-1}y^{-1}(\tilde{\theta}+1),
\end{equation*}
and
\begin{equation*}
\tilde{g}_2(s,y):=\alpha_2 s^{2H-1}\frac{1}{\tilde{\theta}^{\rho}} y^{-\tilde{\theta}\rho+\tilde{\theta}+1}-(H+\epsilon)\tilde{\sigma}s^{2H-1}y^{-1}(\tilde{\theta}+1).
\end{equation*}
We see that 
\begin{align*}
\tilde{g}_1(s,y)&\ge \alpha_{-1}(-\tilde{\theta} y^{2\tilde{\theta}+1})+\alpha_0y^{\tilde{\theta}+1}-\alpha_1\frac{y}{\tilde{\theta}}+\epsilon\tilde{\sigma}y^{-1}(\tilde{\theta}+1)\\
&\ge 0
\end{align*}
for all $s\in (0,1]$ and $y\in(0,y_0)$ for some $y_0>0$. Since $$-\tilde{\theta}\rho+\tilde{\theta}+1<-\frac{1}{H}+1<-1,$$ we also find some $y_1>0$ such that
\begin{align*}
\tilde{g}_2(s,y)&=s^{2H-1} y^{-\tilde{\theta}\rho+\tilde{\theta}+1}\Big(\alpha_2\frac{1}{\tilde{\theta}^{\rho}}-(H+\epsilon)\tilde{\sigma}(\tilde{\theta}+1)y^{\tilde{\theta}\rho-\tilde{\theta}-2}\Big)\\
&\ge h_1s^{2H-1}y^{-\alpha}
\end{align*}
for all $s\in (0,1]$ and $y\in(0,y_1]$, where $h_1>0$ and $\alpha:=\tilde{\theta}\rho-\tilde{\theta}-1$.
So,
\begin{align*}
\tilde{f}(s,y)\ge h_1s^{2H-1}y^{-\alpha}
\end{align*}
for all $s\in (0,1]$, $y\in(0,y_1)$  for some $y_1>0$, which shows that $\tilde{f}$ satisfies (A2). As for (A3), we see that there exists some $y_2\ge 1$ such that 
\begin{align*}
\tilde{f}(s,y)&\le s^{2H-1}\Big(\alpha_2\frac{1}{\tilde{\theta}^{\rho}}y^{-\tilde{\theta}\rho+\tilde{\theta}+1}-H\tilde{\sigma}y^{-1}(\tilde{\theta}+1)\Big)\le h_2s^{2H-1}(1+y)
\end{align*}
for all $s\in (0,1]$, $y\in(y_2,\infty)$ and some $h_2>0$. We have that 
\begin{align*}
\tilde{f}^{\prime}(s,y)=f_1(s,y)+f_2(s,y),
\end{align*}
where 
\begin{align*}
f_1(s,y):=-\alpha_{-1}\tilde{\theta}(2\tilde{\theta}+1)y^{2\tilde{\theta}}+\alpha_0(\tilde{\theta}+1)y^{\tilde{\theta}}-\frac{\alpha_1}{\tilde{\theta}}
\end{align*}
and 
\begin{align*}
f_2(s,y):= s^{2H-1}\Big(\alpha_2\frac{1}{\tilde{\theta}^{\rho}}(-\tilde{\theta}\rho+\tilde{\theta}+1)y^{-\tilde{\theta}\rho+\tilde{\theta}}+H\tilde{\sigma}(\tilde{\theta}+1)y^{-2}\Big).
\end{align*}
So, there exist $y_1,y_2>0$ such that 
\begin{align*}
\tilde{f}'(s,y)\le f_1(s,y)\le K\le s^{2H-1}K=K_s
\end{align*}
for all $s\in (0,1]$, $y\in(0,y_1)$ as well as 
\begin{align*}
\tilde{f}'(s,y)\le f_2(s,y)\le s^{2H-1}K=K_s
\end{align*}
for all $s\in(0,1]$, $y\in(y_2,\infty)$ and some $K>0$. On the other hand, we see that
\begin{align*}
\tilde{f}'(s,y)\le K_1+s^{2H-1}K_2\le s^{2H-1}K=K_s
\end{align*}
for all $s\in (0,1]$, $y_0\in[y_1,y_2]$ for some $K_1,K_2,K>0$. Altogether, we see that $\tilde{f}$ also satisfies (A1). Since $-B^H_{\bigcdot}$ is a fractional Brownian motion, the proof follows.
\end{proof}

\section{Malliavin differentiability and existence of higher moments of solutions}\label{sec5}
In this section, we want to show that the solution $x$ to the SDE
\begin{equation}\label{10}
x_t=x+\int_0^t \tilde{f}(s,x_s)ds-\tilde{\sigma}B_t^H, \quad 0\le t\le 1, x>0,
\end{equation}
is Malliavin differentiable in the direction of $B_{\bigcdot}^H$ for $H\in(0,\frac{1}{2})$. Furthermore, we verify that solutions $x_t$ to \eqref{10} belong to $L^q$ for all $q\ge 1$. For this purpose, let $\tilde{f}_n:(0,1]\times \mathbb{R}\rightarrow \mathbb{R}$, $n\ge 1$ be a sequence of bounded, globally Lipschitz continuous (and smooth) functions such that
\begin{enumerate}
\item $\tilde{f}\mid_{[\frac{1}{n},n]}=\tilde{f}\mid_{(0,1]\times [\frac{1}{n},n]}$ for all $n\ge 1$,\\
\item $\tilde{f}'_n(s,x)\le K_s$ for all $(s,x)\in (0,1]\times \mathbb{R}$, $n\ge 1$, where $K_s$ is defined in (A1).
\end{enumerate}
So we see that
\begin{equation*}
\tilde{f}'_n(s,x)\underset{n\rightarrow \infty}{\longrightarrow}\tilde{f}(s,x)
\end{equation*}
for all $(s,x)\in (0,1]\times (0,\infty)$. Denote by $D^H_{\bigcdot}$ and $D_{\bigcdot}$, the Malliavian derivative in the direction of $B^H_{\bigcdot}$ and $W_{\bigcdot}$, respectively. Here $W_{\bigcdot}$ is the Wiener process with respect to the representation
\begin{equation}\label{new0}
B^H_t=\int^t_0K_H(t,s)dW_s,\quad t\ge 0.
\end{equation}
See the Appendix. Since $-B^H_{\bigcdot}$ is a fractional Brownian motion, let us without loss of generality assume in \eqref{10} that $\tilde{\sigma}=-1$. Because of the regularity of the functions $\tilde{f}_n$, $n\ge 1$, we find that the solutions $x^n_{\bigcdot}$ to 
\begin{equation*}
x^n_t=x+\int_0^t \tilde{f}_n(s,x_s)ds+B_t^H, \quad x>0, \quad 0\le t\le 1
\end{equation*}
are Malliavin differentiable with Malliavin derivative $D_u^Hx_t$ satisfying the equation
\begin{equation*}
D_u^Hx_t^n=\int^t_u \tilde{f}^{\prime}_n(s,x^n_s)D_u^Hx_s^nds+\chi_{[0,t]}(u).
\end{equation*}
Hence,
\begin{equation*}
D_u^Hx^n_t=\chi_{[0,t]}(u)\exp\Big(\int_u^t\tilde{f}_n^{\prime}(s,x_s^n )ds \Big)\quad \lambda\times \text{P-a.e.}
\end{equation*}
for all $0\le t\le 1$ ( $\lambda$ Lebesgue measure). See \cite{Nualart}.
Further, using the transfer principle between $D^H_{\bigcdot}$ and $D_{\bigcdot}$, we have that
\begin{equation}\label{new1}
K_H^*D^H_{\bigcdot}x_t=D_{\bigcdot}x_t
\end{equation}
where $K_H^*:\mathcal{H}\rightarrow L^2([0,T])$ is given by
\begin{equation}
(K_H^*y)(s)=K_H(T,s)y(s)+\int_s^T(y(t)-y(s))\frac{\partial}{\partial t}K_H(t,s)dt
\end{equation}
for
\begin{equation}
\frac{\partial}{\partial t}K_H(t,s)=c_H\Big(H-\frac{1}{2}\Big)\Big(\frac{1}{2}\Big)^{H-\frac{1}{2}}\Big(t-s\Big)^{H-\frac{3}{2}}.
\end{equation}
Here $\mathcal{H}=I_{T^-}^{\frac{1}{2}-H}(L^2)$. See the Appendix. On the other hand, using \eqref{new1}, we also see that
\begin{equation}
D_ux^n_t=\int_u^t \tilde{f}'_n(s,x_s^n)D_ux_s^nds+K_H(t,u)
\end{equation}
in $L^2([0,t]\times \Omega)$ for all $0\le t\le 1$. Set
\begin{equation*}
Y_t^n(u)=D_ux_t^n-K_H(t,u).
\end{equation*}
Then, 
\begin{equation*}
Y_t^n(u)=\int_u^t\Big\{\tilde{f}'_n(s,x_s^n)Y^n_s(u)+\tilde{f}_n(s,x_s^n)K_H(s,u) \Big\}ds.
\end{equation*}
Using the fundamental solution of the equation
\begin{equation*}
\dot{\Phi}(t)=\tilde{f}'_n(t,x_t^n)\cdot\Phi(t),\quad \Phi(u)=1.
\end{equation*}
We then obtain that 
\begin{equation*}
Y_t^n(u)=\int_u^t\exp\Big(\int_s^t\tilde{f}'(r,x_r^n)dr\Big)\tilde{f}'_n(s,x_s^n)K_H(s,u)ds.
\end{equation*}
Hence,
\begin{align*}
D_ux_t^n&=\int^t_u \exp\Big(\int_s^t\tilde{f}'(r,x_r^n)dr\Big)\tilde{f}'_n(s,x_s^n)K_H(s,u)ds+K_H(t,u)\\
&= J_1^n(t,u)+J_2^n(t,u)+K_H(t,u), \quad u<t,\quad \lambda\times\text{P-a.e.},
\end{align*}
where
\begin{align*}
J_1^n(t,u)&:=\int_u^t\exp\Big(\int_s^t \tilde{f}'(r,x_r)dr\Big)\Big( \tilde{f}'_n(s,x_s^n)-K_s\Big)K_H(s,u)ds
\end{align*}
and
\begin{align*}
J_2^n(t,u)&:=\int_u^t\exp\Big(\int_s^t \tilde{f}'(r,x_r)dr\Big)K_s\cdot K_H(s,u)ds.
\end{align*}
Without loss of generality, let $T=t=1$. Then 
\begin{equation}
\int_0^1(D_ux_1^n)^2du\le C\Big\{\int_0^1(J^n_1(1,u))^2du+\int_0^1(J^n_2(1,u))^2du+\int_0^1(K_H(1,u))^2du \Big\}.
\end{equation}
Using Fubini's theorem, we get that 
\begin{align*}
\int_0^1(J^n_1(1,u))^2du&=\int_0^1\Big(\int_0^1\chi_{_{[u,1]}}(s)\exp\Big(\int_s^t \tilde{f}'(r,x_r)dr\Big)\Big( \tilde{f}'_n(s,x_s^n)-K_s\Big)K_H(s,u)ds\Big)^2du\\
&=\int_0^1\int_0^1 \Big\{\exp\Big(\int_{s_1}^1\tilde{f}'(r,x_r)dr\Big)\Big( \tilde{f}'_n(s_1,x_{s_1}^n)-K_{s_1}\Big)\Big)\\
&\times \exp\Big(\int_{s_2}^1\tilde{f}'(r,x_r)dr\Big)\Big( \tilde{f}'_n(s_2,x_{s_2}^n)-K_{s_2}\Big)\int_0^{s_1\wedge s_2} K_H(s_1,u)K_H(s_2,u)du\Big)\Big\}ds_1ds_2.
\end{align*}
From \eqref{new0}, we see for the covariance function 
\begin{align*}
R_H(s_1,s_2)=\mathbb{E}[B_{s_1}^H\cdot B_{s_2}^H]
\end{align*}
that 
\begin{align*}
R_H(s_1,s_2)=\int_0^{s_1\wedge s_2}K_H(s_1,u)K_H(s_2,u)du.
\end{align*}
Since 
\begin{align*}
0\le R_H(s_1,s_2)=\frac{1}{2}(s_1^{2H}+s_2^{2H}-\vert s_1-s_2\vert^{2H})\le 1,\quad H< \frac{1}{2}
\end{align*}
and 
\begin{align*}
\Big( \tilde{f}'_n(s_1,x_{s_1}^n)-K_{s_1}\Big)\cdot \Big( \tilde{f}'_n(s_2,x_{s_2}^n)-K_{s_2}\Big)\ge 0
\end{align*}
for $0< s_1,s_2\le 1$, we find that 
\begin{align*}
\int_0^1(J^n_1(1,u))^2du&\le \Big(\int_0^1\Big(\exp\Big(\int_s^t \tilde{f}'(r,x_r)dr\Big)\Big( \tilde{f}'_n(s,x_s^n)-K_s\Big)ds\Big)^2\\
&=\Big\{-\exp\Big(\int_s^1\tilde{f}'(r,x_r)dr\Big)\Big\vert_{s=0}^1-\int_0^1K_s\exp\Big(\int_s^1\tilde{f}'(r,x_r)dr\Big)ds\Big\}^2\\
&\le \Big(\exp(\int_0^1K_rdr)+\int_0^1K_sds\cdot\exp(\int_0^1K_rdr)\Big)^2.
\end{align*}
Similarly, we also obtain that
\begin{align*}
\int_0^1(J^n_2(1,u))^2du&\le C(K,H)
\end{align*}
for a constant $C(K,H)<\infty$. We also have that 
\begin{align*}
\int_0^1(K_H(1,u))^2du=\mathbb{E}[(B_1^H)^2]=1.
\end{align*}
Altogether, we get that 
\begin{equation}\label{boom}
\mathbb{E}[\int_0^1(D_ux_1^n)^2du]\le C(K,H)
\end{equation}
for all $n\ge 1$ for a constant $C(K,H)<\infty$. Define now the stopping times $\tau_n$ by
\begin{align*}
\tau_n=\inf\Big\{0\le t\le 1; x_t\notin [\frac{1}{n},n]\Big\}\quad (\inf\emptyset=\infty)
\end{align*}
Then we know from the proof of the existence of solutions in the previous section that $\tau_n\nearrow \infty$ for $n\rightarrow\infty$. So,
\begin{align*}
x^n_{t\wedge \tau_n}-x_{t\wedge \tau_n}&=\int_0^{t\wedge \tau_n}\Big\{\tilde{f}_n(s,x_s^n)-\tilde{f}(s,x_s)\Big\}ds\\
&=\int_0^t \chi_{_{[0,\tau_n)}}(s)\Big\{ \tilde{f}_n(s,x_{s\wedge\tau_n}^n)-\tilde{f}_n(s,x_{s\wedge\tau_n})\Big\}ds.
\end{align*}
Hence,
\begin{align*}
\vert x^n_{t\wedge \tau_n}-x_{t\wedge \tau_n}\vert &\le K_n\int^t_0\vert
x^n_{s\wedge \tau_n}-x_{s\wedge \tau_n}\vert ds
\end{align*}
for a Lipschitz constant $K_n$. Then Gronwall's Lemma implies that
\begin{equation*}
x^n_{t\wedge \tau_n}=x_{t\wedge \tau_n}
\end{equation*}
for all $t,n$ P-a.e. Since $\tau_n\nearrow \infty$ for $n\rightarrow\infty$ a.e. we have that
\begin{equation}\label{boom1}
x_t^n\underset{n\rightarrow \infty}{\rightarrow} x_t
\end{equation}
for all t P-a.e. Using the Clark-Ocone formula (see \cite{Nualart}), we get that 
\begin{align*}
x^n_1=\mathbb{E}[x^n_1]+\int_0^1\mathbb{E}[D_sx^n_1\vert \mathcal{F}_s]dW_s,
\end{align*}
where $\{\mathcal{F}\}_{0\le t\le 1}$ is the filtration generated by $W_{\bigcdot}$. It follows that 
\begin{align*}
\mathbb{E}[(x_1^n-\mathbb{E}[x^n_1])^2]
&=\mathbb{E}\Big[\int_0^1\Big(\mathbb{E}[D_sx^n_1\vert \mathcal{F}_s]\Big)^2ds\Big]\\
&\le\mathbb{E}\Big[\int_0^1\mathbb{E}[(D_sx^n_1)^2\vert \mathcal{F}_s]ds\Big]=\int_0^1\mathbb{E}[(D_sx^n_1)^2]ds.
\end{align*}
So, we see from \eqref{boom} that
\begin{align*}
\mathbb{E}[(x_1^n-\mathbb{E}[x^n_1])^2]\le C(K,H)< \infty.
\end{align*}
for all $n\ge 1$.  We also have that
\begin{align*}
\big\vert\vert x_1^n-\mathbb{E}[x^n_1]\vert-\vert x_1-\mathbb{E}[x^n_1]\vert \big\vert\le \vert x^n_1-x_1\vert\underset{n\rightarrow \infty}{\longrightarrow} 0
\end{align*}
because of \eqref{boom1}.  So,
\begin{align*}
\underset{n\rightarrow \infty}{\varliminf}\vert x_1^n-\mathbb{E}[x^n_1]\vert=\underset{n\rightarrow \infty}{\varliminf}\vert x_1-\mathbb{E}[x^n_1]\vert.
\end{align*}
Suppose that $\mathbb{E}[x_1^n]$, $n\ge 1$ is unbounded. Then there exists a subsequence $n_k$, $k\ge 1$ such that 
\begin{align*}
\vert\mathbb{E}[x^{n_k}_1]\vert\underset{n\rightarrow \infty}{\longrightarrow} \infty.
\end{align*}
It follows from the Lemma of Fatou and the positivity of $x_t$ that 
\begin{align*}
\infty&=\mathbb{E}\Big[\underset{k\rightarrow \infty}{\varliminf}\Big(\big\vert x_1-\vert\mathbb{E}[x^{n_k}_1]\vert\big\vert\Big)^2\Big]\\
&\le \mathbb{E}\Big[\underset{k\rightarrow \infty}{\varliminf}\Big(\big\vert x_1-\mathbb{E}[x^{n_k}_1]\big\vert\Big)^2\Big]\\
&=\mathbb{E}\Big[\underset{k\rightarrow \infty}{\varliminf}\Big(\big\vert x^{n_k}_1-\mathbb{E}[x^{n_k}_1]\big\vert\Big)^2\Big]\\
&\le \underset{k\rightarrow \infty}{\varliminf}\mathbb{E}\Big[\big\vert x^{n_k}_1-\mathbb{E}[x^{n_k}_1]\big\vert^2\Big]\le C<\infty,
\end{align*}
which is a contradiction. Hence,
\begin{align*}
\sup_{n\ge 1}\vert \mathbb{E}[x_1^n]\vert<\infty.
\end{align*}
Further, we also obtain from the Burkholder-Davis-Gundy inequality and \eqref{boom} that
\begin{align}\label{estim}
\mathbb{E}[\vert x_1^n\vert^{2p}]&<C_p\Big(\vert\mathbb{E}[x_1^n]\vert^{2p}+\mathbb{E}\Big[\big(\int_0^1 \mathbb{E}[D_sx_1^n\vert \mathcal{F}_s]dW_s\big)^{2p}\Big] \Big)\nonumber\\
&\le C_p\Big(\vert\mathbb{E}[x_1^n]\vert^{2p}+\mathbb{E}\Big[\big(\sup_{0\le u\le 1}\big\vert\int_0^u \mathbb{E}[D_sx_1^n\vert \mathcal{F}_s]dW_s\big\vert\big)^{2p}\Big]\Big)\nonumber\\
&\le C_p\Big(\vert\mathbb{E}[x_1^n]\vert^{2p}+m_p\mathbb{E}\Big[(\int_0^1 \mathbb{E}[D_sx_1^n\vert \mathcal{F}_s])^2ds\big)^{p}\Big]\Big)\nonumber\\
&\le C(p,K,H)
\end{align}
for $n\ge 1$. So it follows from \eqref{boom1} and the Lemma of Fatou that
\begin{align*}
\mathbb{E}[\vert x_1\vert^{2p}]\le \underset{n\rightarrow \infty}{\varliminf}\mathbb{E}[\vert x_1^n\vert]^{2p}\le C(p,K,H)<\infty
\end{align*}
for all $p\ge 1$.
So we obtain the following result:
\begin{theorem}\label{highermoment}
Let $x_t, 0\le t\le 1$ be the solution to \eqref{10}. Then $x_t\in L^q(\Omega)$ for all $q\ge 1$ and $0\le t\le 1$.
\end{theorem}
In addition, we obtain from Lemma 1.2.3 in \cite{Nualart} in connection with the estimate \eqref{estim} that $x_1$ is Malliavin differentiable in the direction of $W_{\bigcdot}$. The latter, in combination with \eqref{new1}, also entails the Malliavin differentiability of $x_1$ with respect to $B^H_{\bigcdot}.$ Thus we have  also shown the following result:
\begin{theorem}\label{Thrm 4.2}
The positive unique strong solution $x_t$ to \eqref{10} is Malliavin differentiable in the direction of $B^H_{\bigcdot}$ and $W_{\bigcdot}$ for all $0\le t\le 1$.
\end{theorem}

\section{Application}\label{sec6}
In this section, we aim at applying the results of the previous section to obtain a unique strong solution to $x_t$ to the SDE 
\begin{equation}\label{app1}
x_t=x_0+\int_0^t\big(\alpha_{-1}x^{-1}_s-\alpha_0+\alpha_1x_s-\alpha_2s^{2H-1}x_s^{\rho}\big)ds+\int_0^t\sigma x_s^{\theta}dB_s^H,
\end{equation}
$0\le t\le 1$, for $H\in(\frac{1}{3},\frac{1}{2})$, $\tilde{\theta}>0$, $\rho>1+\frac{1}{H\tilde{\theta}}$, $\sigma>0$ and $\theta:=\frac{\tilde{\theta}+1}{\tilde{\theta}}$. Here, the stochastic integral with respect to $B_{\bigcdot}^H$ is defined by
\begin{equation}\label{app2}
\int_0^tg(x_s)dB^H_s=\int_0^t-Hs^{2H-1}g^{\prime}(x_s)ds+\int_0^tg(x_s)d^{\circ}B_s^H
\end{equation}
for functions $g\in \mathcal{C}^3$. See also the second Remark \ref{appremark} below. The stochastic integral on the right hand side of \eqref{app2} is the symmetric integral with respect to $B^H_{\bigcdot}$ introduced by F. Russo, P. Vallois. See e.g.  \cite{Errami} and the references therein. Such an integral denoted by 
\begin{equation}
\int_0^tY_sd^{\circ}X_s, \quad t\in [0,1]
\end{equation}
for continuous process $X_{\bigcdot}$, $Y_{\bigcdot}$ is defined as
\begin{equation*}
\lim_{\epsilon\searrow 0}\frac{1}{2\epsilon}\int_0^tY_s(X_{s+\epsilon}-X_s)ds,
\end{equation*}
provided this limit exists in the ucp-topology. In order to construct a solution to \eqref{app1}, we need a version of the It\^{o} formula for processes $Y_{\bigcdot}$, which have a finite cubic variation. A continuous process is said to have a finite strong cubic variation (or 3-variation), denoted by $[Y,Y,Y]$, if 
\begin{equation*}
[Y,Y,Y]:=\lim_{\epsilon\searrow 0}\frac{1}{\epsilon}\int_0^t(Y_{s+\epsilon}-Y_s)^3ds
\end{equation*}
exists in ucp as well as 
\begin{equation*}
\sup_{0< \epsilon\le 1}\frac{1}{\epsilon}\int_0^1(Y_{s+\epsilon}-Y_s)^3ds< \infty\quad \text{a.e.}
\end{equation*}
See \cite{Errami}. Using the concept of finite strong cubic variation, one can show the following It\^{o} formula (see \cite{Errami}).
\begin{theorem}\label{apptheorem}
Assume that $Y_{\bigcdot}$ is a real valued process with finite strong cubic variation and $g\in \mathcal{C}^3$. Then 
\begin{equation*}
g(Y_t)=g(Y_0)+\int_0^tg^{\prime}(Y_s)d^{\circ}Y_s-\frac{1}{12}\int_0^tg^{\prime\prime\prime}(Y_s)d[Y,Y,Y]_s, \quad 0\le t\le 1.
\end{equation*}
\end{theorem}
\begin{remark}
The last term on the right hand side of the equation is a Lebesgue-Stieltjes integral with respect to the bounded variation process $[Y,Y,Y]$.
\end{remark}
\begin{remark}\label{appremark}\leavevmode
\begin{itemize}
\item We mention that for $Y_{\bigcdot}=B_{\bigcdot}^H$, $H\in(\frac{1}{3},\frac{1}{2})$, $[B_{\bigcdot}^H,B_{\bigcdot}^H,B_{\bigcdot}^H]$ is zero a.e.
\item If $X_{\bigcdot}=B_{\bigcdot}^H$ in \eqref{app2}, then it follows from Theorem 6.3.1 in \cite{Biagini} that our stochastic integral in \eqref{app2} equals the Wick-It\^{o}-Skorohod integral. The latter also gives a justification for the definition of the stochastic integral in \eqref{app2} in the general case.
\end{itemize}
\end{remark}
\begin{theorem}\label{end}
Suppose that $H\in(\frac{1}{3},\frac{1}{2})$, $\tilde{\theta}>0$, $\sigma>0$ and $\rho>1+\frac{1}{H\tilde{\theta}}$. Let $\theta=\frac{\tilde{\theta}+1}{\tilde{\theta}}$. Then there exists a unique strong and positive solution to the \textup{SDE} \eqref{app1}.
\end{theorem}
\begin{proof}
Let $y_{\bigcdot}$ be the unique strong and positive solution to
\begin{equation*}
y_t=x+\int_0^t\tilde{f}(s,y_s)ds-\tilde{\sigma}B_t^H, \quad 0\le t\le 1,\quad x>0,
\end{equation*}
where $\tilde{f}$ is defined as in Section \ref{sec3}. Define $g\in \mathcal{C}^3$ by $g(y)=\frac{y^{-\tilde{\theta}}}{\tilde{\theta}}$. Then Theorem \ref{apptheorem} entails that 
\begin{equation*}
x_t:=g(y_t)=\frac{x^{-\tilde{\theta}}}{\tilde{\theta}}+\int_0^t(-1)y_s^{-(\tilde{\theta}+1)}d^{\circ}y_s-\frac{1}{12}\int_0^tg^{\prime\prime\prime}(y_s)d[y,y,y]_s.
\end{equation*}
Since $[B_{\bigcdot}^H,B_{\bigcdot}^H,B_{\bigcdot}^H]$ is zero a.e. (see Remark \ref{appremark}), we observe that $[y,y,y]$ is zero a.e. So
\begin{align*}
x_t&=\frac{x^{-\tilde{\theta}}}{\tilde{\theta}}+\int_0^t(-1)y_s^{-(\tilde{\theta}+1)}d^{\circ}y_s\\
&=\frac{x^{-\tilde{\theta}}}{\tilde{\theta}}+\int_0^t(-1)y_s^{-(\tilde{\theta}+1)}\tilde{f}(s,y_s)ds+\int_0^t\tilde{\sigma} y_s^{-(\tilde{\theta}+1)}d^{\circ} B_s^H\\
&=\frac{x^{-\tilde{\theta}}}{\tilde{\theta}}-\int_0^t\Big\{(-1)y_s^{-(\tilde{\theta}+1)}\tilde{f}(s,y_s)-H\tilde{\sigma} s^{2H-1}y_s^{-(\tilde{\theta}+2)}(\tilde{\theta}+1)\Big\}ds+\int_0^t\tilde{\sigma} y_s^{-(\tilde{\theta}+1)}dB_s^H.
\end{align*}
Since we can write $(y_s)^{-(\tilde{\theta}+1)}=\tilde{\theta}^{\theta}\Big(\frac{y_s^{-\tilde{\theta}}}{\tilde{\theta}}\Big)^{\theta}$, we now have
\begin{align*}
x_t&=\frac{x^{-\tilde{\theta}}}{\tilde{\theta}}+\int_0^tf(s,\frac{y_s^{-\tilde{\theta}}}{\tilde{\theta}})ds+\int_0^t\tilde{\sigma} y_s^{-(\tilde{\theta}+1)}dB_s^H\\
&=\frac{x^{-\tilde{\theta}}}{\tilde{\theta}}+\int_0^tf(s,\frac{y_s^{-\tilde{\theta}}}{\tilde{\theta}})ds+\int_0^t\tilde{\sigma} \tilde{\theta}^{\theta}(x_s)^{\theta}dB_s^H,
\end{align*}
where $f(s,y):=\alpha_{-1}y^{-1}-\alpha_0+\alpha_1y-\alpha_2s^{2H-1}y^{\rho}$, $s\in (0,1]$, $y\in(0,\infty)$. So $x_{\bigcdot}$ satisfies the SDE \eqref{app1} if we choose $\tilde{\sigma}=\tilde{\theta}^{-\theta}\sigma$ for $\sigma>0$. In order to show the uniqueness of solutions to SDE \eqref{app1}, one can apply the It\^{o} formula in Theorem \ref{apptheorem} to the inverse function $g^{-1}$ given by $g^{-1}(y)=\big(\tilde{\theta}\big)^{-\frac{1}{\tilde{\theta}}}    y^{-\frac{1}{\tilde{\theta}}}$ by using the fact that $[B_{\bigcdot}^H,B_{\bigcdot}^H,B_{\bigcdot}^H]=0$ a.e. for $H\in(\frac{1}{3},\frac{1}{2})$.
\end{proof}
Finally, using the same arguments as in the proof of Theorem \ref{end}, we also get the following result for the alternative Ait-Sahalia model \eqref{intro:eq:3}:
\begin{theorem}\label{final}
Retain the conditions of Theorem \ref{end} with respect to $H,\tilde{\theta}, \theta$ and $\rho$. Then there exists a unique strong solution $x_t>0$ to \textup{SDE} \eqref{intro:eq:3}.
\end{theorem}
\begin{proof}
Just as in the proof of Theorem \ref{end}, we can consider the SDE \eqref{sec2:eq1}, where the vector field $\tilde{f}$ now is given by
\begin{align}\label{end:eq1}
\tilde{f}(s,y)&=\alpha_{-1}(-\tilde{\theta} y^{2\tilde{\theta}+1})+\alpha_0y^{\tilde{\theta}+1}-\alpha_1\frac{y}{\tilde{\theta}}+\alpha_2 \frac{1}{\tilde{\theta}^{\rho}} y^{-\tilde{\theta}\rho+\tilde{\theta}+1}
\end{align}
for $0<y < \infty$. Then as in the proof of Corollary \eqref{corona} one immediately verifies that $\tilde{f}$ satisfies the assumptions of Theorem \ref{thrm1}, which yields a unique strong solution $y_t>0$ to \eqref{sec2:eq1} in this case. In exactly the same way, we also obtain the results of Theorem \ref{highermoment} and Theorem \ref{Thrm 4.2} with respect to $\tilde{f}$ in \eqref{end:eq1}. Finally, we can apply the It\^{o} formula as in the proof of Theorem \ref{end} and construct a unique strong solution $x_t>0$ to \eqref{intro:eq:3} based on $y_{\bigcdot}$.
\end{proof}

\section{Appendix}
\bigskip For some of the proofs in this article we need to recall some basic concepts from fractional calculus (see 
\cite{Lizorkin} and \cite{Samko}).

Let $a,$ $b\in \mathbb{R}$ with $a<b$. Let $f\in L^{p}([a,b])$ with $p\geq 1$
and $\alpha >0$. Then the \emph{left-} and \emph{right-sided
Riemann-Liouville fractional integrals} are defined as
\begin{equation*}
I_{a^{+}}^{\alpha }f(x)=\frac{1}{\Gamma (\alpha )}\int_{a}^{x}(x-y)^{\alpha
-1}f(y)dy
\end{equation*}%
and 
\begin{equation*}
I_{b^{-}}^{\alpha }f(x)=\frac{1}{\Gamma (\alpha )}\int_{x}^{b}(y-x)^{\alpha
-1}f(y)dy
\end{equation*}%
for almost all $x\in \lbrack a,b]$. Here $\Gamma $ is the Gamma
function.

Let $p\geq 1$ and let $I_{a^{+}}^{\alpha }(L^{p})$ (resp. $I_{b^{-}}^{\alpha
}(L^{p})$) be the image of $L^{p}([a,b])$ of the operator $I_{a^{+}}^{\alpha
}$ (resp. $I_{b^{-}}^{\alpha }$). If $f\in I_{a^{+}}^{\alpha }(L^{p})$
(resp. $f\in I_{b^{-}}^{\alpha }(L^{p})$) and $0<\alpha <1$ then we can
define the \emph{left-} and \emph{right-sided Riemann-Liouville
fractional derivatives} by 
\begin{equation*}
D_{a^{+}}^{\alpha }f(x)=\frac{1}{\Gamma (1-\alpha )}\frac{d}{dx}\int_{a}^{x}%
\frac{f(y)}{(x-y)^{\alpha }}dy
\end{equation*}%
and 
\begin{equation*}
D_{b^{-}}^{\alpha }f(x)=\frac{1}{\Gamma (1-\alpha )}\frac{d}{dx}\int_{x}^{b}%
\frac{f(y)}{(y-x)^{\alpha }}dy.
\end{equation*}

The left- and right-sided derivatives of $f$ can be represented as
\begin{equation*}
D_{a^{+}}^{\alpha }f(x)=\frac{1}{\Gamma (1-\alpha )}\left( \frac{f(x)}{%
(x-a)^{\alpha }}+\alpha \int_{a}^{x}\frac{f(x)-f(y)}{(x-y)^{\alpha +1}}%
dy\right)
\end{equation*}%
and 
\begin{equation*}
D_{b^{-}}^{\alpha }f(x)=\frac{1}{\Gamma (1-\alpha )}\left( \frac{f(x)}{%
(b-x)^{\alpha }}+\alpha \int_{x}^{b}\frac{f(x)-f(y)}{(y-x)^{\alpha +1}}%
dy\right) .
\end{equation*}

The above definitions imply that
\begin{equation*}
I_{a^{+}}^{\alpha }(D_{a^{+}}^{\alpha }f)=f
\end{equation*}%
for all $f\in I_{a^{+}}^{\alpha }(L^{p})$ and 
\begin{equation*}
D_{a^{+}}^{\alpha }(I_{a^{+}}^{\alpha }f)=f
\end{equation*}%
for all $f\in L^{p}([a,b])$ and similarly for $I_{b^{-}}^{\alpha }$ and $%
D_{b^{-}}^{\alpha }$.

\bigskip

Denote by $B^{H}=\{B_{t}^{H},t\in \lbrack 0,T]\}$ a $d$-dimensional \emph{%
fractional Brownian motion} with Hurst parameter $H\in (0,\frac{1}{2})$. The later means that $B_{\cdot }^{H}$ is a centered Gaussian process with a covariance
function given by 
\begin{equation*}
(R_{H}(t,s))_{i,j}:=E[B_{t}^{H,(i)}B_{s}^{H,(j)}]=\delta _{ij}\frac{1}{2}%
\left( t^{2H}+s^{2H}-|t-s|^{2H}\right) ,\quad i,j=1,\dots ,d,
\end{equation*}%
where $\delta _{ij}$ is one, if $i=j$, or zero else.

In the sequel, we also shortly recall the construction of the fractional
Brownian motion, which can be found in \cite{Nualart}. For convenience, we
restrict ourselves to the case $d=1$.

Denote by $\mathcal{E}$ the class of step functions on $[0,T]$ and Let $%
\mathcal{H}$ be the Hilbert space which one gets through the completion of $%
\mathcal{E}$ with respect to the inner product 
\begin{equation*}
\langle 1_{[0,t]},1_{[0,s]}\rangle _{\mathcal{H}}=R_{H}(t,s).
\end{equation*}%
The latter provides an extension of the mapping $1_{[0,t]}\mapsto B_{t}$ to an
isometry between $\mathcal{H}$ and a Gaussian subspace of $L^{2}(\Omega )$
with respect to $B^{H}$. Let $\varphi \mapsto B^{H}(\varphi )$ be this
isometry.

If $H<\frac{1}{2}$, one finds that the covariance function $R_{H}(t,s)$ can be
represented as

\bigskip\ 
\begin{equation}
R_{H}(t,s)=\int_{0}^{t\wedge s}K_{H}(t,u)K_{H}(s,u)du,  \label{2.2}
\end{equation}%
where 
\begin{equation}
K_{H}(t,s)=c_{H}\left[ \left( \frac{t}{s}\right) ^{H-\frac{1}{2}}(t-s)^{H-%
\frac{1}{2}}+\left( \frac{1}{2}-H\right) s^{\frac{1}{2}-H}\int_{s}^{t}u^{H-%
\frac{3}{2}}(u-s)^{H-\frac{1}{2}}du\right] .  \label{KH}
\end{equation}%
Here $c_{H}=\sqrt{\frac{2H}{(1-2H)\beta (1-2H,H+\frac{1}{2})}}$ and $\beta $ is the
Beta function. See \cite[Proposition 5.1.3]{Nualart}.

Using the kernel $K_{H}$, one can obtain via (\ref{2.2}) an isometry $%
K_{H}^{\ast }$ between $\mathcal{E}$ and $L^{2}([0,T])$ such that $%
(K_{H}^{\ast }1_{[0,t]})(s)=K_{H}(t,s)1_{[0,t]}(s).$ This isometry allows
for an extension to the Hilbert space $\mathcal{H}$, which has the following
representations in terms of fractional derivatives

\begin{equation*}
(K_{H}^{\ast }\varphi )(s)=c_{H}\Gamma \left( H+\frac{1}{2}\right) s^{\frac{1%
}{2}-H}\left( D_{T^{-}}^{\frac{1}{2}-H}u^{H-\frac{1}{2}}\varphi (u)\right)
(s)
\end{equation*}%
and 
\begin{align*}
(K_{H}^{\ast }\varphi )(s)=& \,c_{H}\Gamma \left( H+\frac{1}{2}\right)
\left( D_{T^{-}}^{\frac{1}{2}-H}\varphi (s)\right) (s) \\
& +c_{H}\left( \frac{1}{2}-H\right) \int_{s}^{T}\varphi (t)(t-s)^{H-\frac{3}{%
2}}\left( 1-\left( \frac{t}{s}\right) ^{H-\frac{1}{2}}\right) dt.
\end{align*}%
for $\varphi \in \mathcal{H}$. One can also prove that $\mathcal{H}%
=I_{T^{-}}^{\frac{1}{2}-H}(L^{2})$. See \cite{DU} and \cite[Proposition 6]%
{alos}.

We know that $K_{H}^{\ast }$ is an isometry from $\mathcal{H}$ into $%
L^{2}([0,T])$. Thus, the $d$-dimensional process $W=\{W_{t},t\in \lbrack
0,T]\}$ defined by 
\begin{equation}
W_{t}:=B^{H}((K_{H}^{\ast })^{-1}(1_{[0,t]})) \label{WBH}
\end{equation}%
is a Wiener process and the process $B^{H}$ has the representation 
\begin{equation}
B_{t}^{H}=\int_{0}^{t}K_{H}(t,s)dW_{s}.  \label{BHW}
\end{equation}%

\end{document}